\documentclass[reqno,12pt]{amsart}
\usepackage{amsmath, amsthm, amssymb, booktabs}
\usepackage{mathrsfs}

\advance \topmargin by -\headheight
\advance \topmargin by -\headsep
     
\evensidemargin \oddsidemargin
\newtheorem{theorem}{Theorem}[section]
\newtheorem{lemma}[theorem]{Lemma}

\theoremstyle{definition}

\theoremstyle{remark}

\numberwithin{equation}{section}

\newcommand{\mmod}[1]{\,\,(\text{mod}\,\,#1)}

\def\calA{{\mathscr A}}

\def\grN{{\mathfrak N}}

 \def\Del{{\Delta}}
  
 \def\Eta{{\mathrm H}}

\def\Ups{{\Upsilon}} 

\def\ome{{\omega}} \def\Ome{{\Omega}} 

\def\eps{\varepsilon}

\def\le{\leqslant} \def\ge{\geqslant}

\begin{document}
\title[Partitio Numerorum]{Partitio Numerorum: sums of a prime\\ and a number of $k$-th 
powers}
\author[J\"org Br\"udern]{J\"org Br\"udern}
\address{Mathematisches Institut, Bunsenstrasse 3--5, D-37073 G\"ottingen, Germany}
\email{jbruede@gwdg.de}
\author[Trevor D. Wooley]{Trevor D. Wooley}
\address{Department of Mathematics, Purdue University, 150 N. University Street, West 
Lafayette, IN 47907-2067, USA}
\email{twooley@purdue.edu}
\subjclass[2010]{11P55, 11P05, 11P32}
\keywords{Waring problem, Goldbach problem, Hardy-Littlewood method.}
\thanks{First Author supported by Deutsche Forschungsgemeinschaft Project Number 
255083470. Second author supported by NSF grants DMS-1854398 and DMS-2001549.}
\date{}

\begin{abstract} Let $k$ be a natural number and let $c=2.134693\ldots$ be the unique 
real solution of the equation $2c=2+\log (5c-1)$ in $[1,\infty)$. Then, when $s\ge ck+4$, 
we establish an asymptotic lower bound of the expected order of magnitude for the number 
of representations of a large positive integer as the sum of one prime and $s$ positive 
integral $k$-th powers.
\end{abstract}
\maketitle

\section{Introduction} Although strongly influenced by an earlier contribution of Hardy and 
Ramanujan \cite{HR}, the famous series {\em Partitio Numerorum} written jointly by Hardy 
and Littlewood marks the arrival of the circle method. In the speculative section of part III 
(see \cite{pn3}), they considered representations of natural numbers $n$ as the sum of a 
prime and a number of $k$-th powers. Thus, the equation
\begin{equation}\label{1.1}
n= p+x_1^k + \ldots + x_s^k,
\end{equation}
in which $k\ge 2$ and $s\ge 1$ are given natural numbers, is to be solved in primes $p$ 
and natural numbers $x_j$ $(1\le j\le s)$. Although their conjectures H, J and L are 
concerned with squares and cubes only, it is plain from the discussion that their method 
supports a conjectural asymptotic formula for the number $r(n)=r_{k,s}(n)$ of solutions of 
\eqref{1.1} in general. If this formula were true for $r_{k,1}(n)$, then all sufficiently large 
$n$ for which the polynomial $n-x^k$ is irreducible over the rationals would be the sum of 
a prime and a $k$-th power. The analogous formula for $r_{k,2}(n)$ suggests that all large 
$n$ are the sum of a prime and two $k$-th powers. There is an extensive literature related 
to the case of squares, in which $k=2$ (e.g.~\cite{Halb,S}), culminating in the works of 
Hooley \cite{HooActa, HooTract} and Linnik \cite{L1, L2} that confirm the Hardy-Littlewood 
formula for $r_{2,2}(n)$. Miech \cite{M} showed that the formula for $r_{2,1}(n)$ holds 
for almost all $n$, in the sense that the exceptional $n$ have density zero.\par  

For $k\ge 3$, less is known. There are quantitative results allied to that of Miech counting 
exceptional $n$ where the conjectured formula for $r_{k,1}(n)$ fails (most recently in work 
of Br\"udern \cite{Tsuk}). Other results concern the sparsity of numbers $n$ where 
$r_{k,1}(n)=0$, for example \cite{BP}. We are not aware of noteworthy unconditional 
contributions that relate to the cases $s>1$ of \eqref{1.1}. There are, of course, results 
that follow routinely from mean value estimates for $k$-th power Weyl sums. As a first 
example, we mention the minor arc estimate for the eighth moment of a cubic Weyl sum of 
Vaughan \cite{WPC} in the improved form due to Boklan \cite{Bok}. This is of strength 
sufficient to decouple the prime from the $k$-th powers via Schwarz's inequality in a direct 
application of the Hardy-Littlewood method, and as a result one obtains an asymptotic 
formula for $r_{3,4}(n)$ (see \cite[Theorem 2]{Kaw1996}). For larger exponents $k$, one 
may restrict the variables $x_j$ to be smooth numbers in order to make smaller values of 
$s$ accessible. If one invokes the best currently known mean value estimates for smooth 
Weyl sums (see \cite{Woo1992, Woo1995} or Lemma \ref{lemma2.2} below) and 
decouples the prime as before, then for an explicit quantity $s_0(k)$ satisfying
\[
s_0(k)=\tfrac{1}{2}k(\log k+\log \log k+2+o(1)),
\]
one finds that
\begin{equation}\label{1.2}
r_{k,s}(n) \gg n^{s/k} / \log n  
\end{equation}
whenever $s\ge s_0(k)$. The lower bound \eqref{1.2} is of the order of magnitude that is 
suggested by the hypothetical asymptotic formula.\par

No improvement on the severe condition $s\ge s_0(k)$ is known. Our first theorem shows 
that it suffices for $s$ to grow linearly with $k$. In order to state this result precisely, let 
$c$ be the unique solution of the transcendental equation
\[
2c=2+\log (5c-1)
\]
in the interval $[1,\infty)$. The decimal representation is $c=2.134693\ldots$.

\begin{theorem}\label{thm1.1}
Let $k\in \mathbb N$ and $s\ge ck +4$. Then $r_{k,s}(n) \gg n^{s/k} / \log n$.
\end{theorem}

For small values of $k$ this conclusion is susceptible to some improvement because our 
proof is tuned to perform optimally for very large $k$. Based on the methods of 
\cite{V89, VW1994, VW,W3}, the naive decoupling approach yields \eqref{1.2} for the pairs 
$(k,s)=(4,6)$, $(5,9)$, $(6,12)$ and $(7,16)$, for example. Our method yields improved 
results for $k\ge 6$.

\begin{theorem}\label{thm1.2}
Suppose that $6\le k\le 20$ and $s\ge S_0(k)$, where $S_0(k)$ is determined according to 
the entries of Table \ref{tab1}. Then whenever $s\ge S_0(k)$, one has 
$r_{k,s}(n) \gg n^{s/k}/\log n$.
\end{theorem}

\begin{table}[h]
\begin{center}
\begin{tabular}{ccccccccccccccccc}
\toprule
$k$ & $5$ & $6$ & $7$ & $8$ & $9$ & $10$ & $11$ & $12$ & $13$ & $14$ & $15$ & $16$ & $17$ & $18$ & $19$ & $20$\\
$S_0(k)$ & & $11$ & $13$ & $16$ & $18$ & $20$ & $22$ & $24$ & $27$ & $29$ & $31$ & $33$ & $35$ & $37$ & $40$ & $42$\\
$S_1(k)$ & $8$ & $10$ & $12$ & $14$ & $17$ & $19$ & $21$ & $23$ & $25$ & $27$ & $29$ & $31$ & $32$ & $34$ & $36$ & $38$\\
\bottomrule
\end{tabular}\\[6pt]
\end{center}
\caption{Critical numbers of $k$-th powers for Theorems \ref{thm1.2} and 
\ref{thm1.4}}\label{tab1}
\end{table}

Our approach to equation \eqref{1.1} involves the circle method, and we are therefore 
limited by the familiar square root cancellation barrier. For the case at hand this says that 
the range $s\le k$ is outside the scope of the method unless one is able to explore 
cancellations that get lost after application of the triangle inequality. Thus, we require 
scarcely more than twice as many $k$-th powers relative to the limit of the method, and for 
small $k$ even fewer.\par

Further progress with problems involving primes is often possible if one assumes that no 
Dirichlet $L$-function has a zero in the half-plane $\text{Re}(z)>\frac 12$. This is the 
generalised Riemann hypothesis, abbreviated to GRH hereafter (and referred to by Hardy 
and Littlewood \cite{pn3} as Hypothesis R$^*$). Hooley's original work \cite{HooActa} on 
$r_{2,2}(n)$ initially depended on GRH, and in a joint effort with Kawada \cite{ARTS6} the 
authors deduced from GRH the existence of infinitely many primes  representable as the 
sum of $2\lceil 4k/3\rceil$ positive integral $k$-th powers. The method of \cite{ARTS6} is 
readily adapted to establish the lower bound \eqref{1.2} for $s\ge 2\lceil 4k/3\rceil$, 
subject to GRH, but this is now superseded unconditionally by Theorem \ref{thm1.1}. 
However, assuming the truth of GRH we are able to relax the conditions on $s$ in 
Theorems \ref{thm1.1} and \ref{thm1.2}. In particular, we improve on the naive decoupling 
approach for fifth powers. Our results feature the unique solution of the transcendental 
equation
\[
2c'=2+\log (4c'-1)
\]
in the interval $[1,\infty)$. Its decimal representation is $c'= 1.961969\ldots$.
    
\begin{theorem}\label{thm1.3}
Should GRH be true, then when $k\in \mathbb N$ and $s\ge c'k +4$, one has 
$r_{k,s}(n)\gg n^{s/k}/\log n$.
\end{theorem}

\begin{theorem}\label{thm1.4}
Suppose that $5\le k\le 20$ and $s\ge S_1(k)$, where $S_1(k)$ is determined according to 
the entries of Table \ref{tab1}. Then, should GRH be true and $s\ge S_1(k)$, one has 
$r_{k,s}(n) \gg n^{s/k}/\log n$.
\end{theorem} 

It should be noted that Theorems \ref{thm1.3} and \ref{thm1.4} depend on GRH only in 
the most indirect way, through the exponential sum estimate \eqref{4.2}. In fact, the bound 
\eqref{4.2} is also available for a type II exponential sum with both variables of summation 
near $\sqrt n$. The unconditional bound \eqref{4.1} requires estimates for type II sums 
with the shorter of the two variables of summation ranging over $[n^{2/5},n^{1/2}]$. If 
one waives Vaughan's identity and works with type II sums directly on the level of 
Diophantine equations, then prime detecting sieves typically narrow the range for type II 
sums. This should lead to an  improvement of Theorem \ref{thm1.1}, and perhaps one can 
handle the case $k=5$, $s=8$ unconditionally in this way. The extra complications, however, 
would disguise the simplicity of new elements that we introduce to the circle method here, 
and we therefore refrain from elaborating on this idea in this paper.\par

Our methods apply equally well to questions concerning primes representable as sums of 
positive integral $k$-th powers. If $R(N)=R_{k,s}(N)$ denotes the number of solutions of 
the equation 
\begin{equation}\label{1.3}
p=x_1^k + \ldots + x_s^k
\end{equation}
in primes $p\le N$ and natural numbers $x_j$, then subject to the conditions in Theorems 
\ref{thm1.1}, \ref{thm1.2}, \ref{thm1.3} or \ref{thm1.4}, it can be shewn that 
$R(N)\gg N^{s/k}(\log N)^{-1}$. Hardy and Littlewood \cite{pn3} referred to this class of 
problems as conjugate to those defined by equation \eqref{1.1}, and made them the 
subject of their conjectures M and N. It should be said, though, that the equation 
\eqref{1.3} seems to be somewhat easier than its counterpart \eqref{1.1}. It has been 
known for centuries that primes of the form $4l+1$ are the sum of two squares, yet the 
conjugate problem concerning the numbers that are the sum of two squares and a prime 
was solved by the generation preceding us. More recent developments with prime 
detecting sieves led to remarkable progress in this area. Friedlander and Iwaniec \cite{FI} 
showed that the polynomial $x^2+y^4$ captures its primes. This should be viewed as a 
result concerning sums of two squares in which one of the variables is itself a square. Soon 
afterwards Heath-Brown \cite{HB} found infinitely many primes of the form $x^3+2y^3$, 
thereby confirming conjecture N of Hardy and Littlewood \cite{pn3}. Most recently of all, 
Maynard \cite{May2020} has considered primes represented by more general incomplete 
norm forms. These spectacular results depend, in some way or other, on the homogeneity 
of the polynomial on the right hand side of \eqref{1.3}. It seems that there are significant 
obstacles preventing us from extracting lower bounds for $r_{3,3}(n)$, for example, along 
these lines.\par

The success of our approach depends on two innovations. The first is a new mean value 
estimate for moments of smooth Weyl sums over major arcs. Very recently, Liu and Zhao 
\cite{LZ} obtained such estimates. Their method rests on the large sieve, and in 
consequence the width of an individual arc centered at a Farey fraction should be as small 
as $1/n$ (normalised for applications to equation \eqref{1.1}). Any inflation of this width 
implies an eternal loss, and this is typically not tolerable. In their work, the use of weights 
and the Poisson summation formula makes it possible to control losses, but in applications 
with primes, for example, this does not seem possible. In Lemma \ref{lemma2.3} below we 
describe a result in which the major arcs have their natural shape, and are therefore much 
wider than $1/n$ if the denominator of the Farey fraction at the center is small. 
Nonetheless, our estimate performs just as well as one can expect from 
\cite[Lemma 5.6]{LZ}, but it is easier to use, and in some cases, and in particular in the 
situation considered in this paper, the wider arcs are essential for the success of the 
method. In contrast to \cite[Lemma 5.6]{LZ}, our result neither depends on the the large 
sieve, nor makes reference to Diophantine equations. Thus, again in contrast to the work of 
Liu and Zhao, we are able to handle fractional moments with ease, and we work with 
ordinary smooth Weyl sums, avoiding preseeded primes lying in certain intervals. This is 
important if one wishes to import results depending on breaking convexity devices. 
Therefore, our approach offers several advantages and extra flexibility. The new lemma has 
applications well beyond those presented in this paper. In fact, the argument leading to 
Lemma \ref{lemma2.3} is very direct and simple in spirit. As we shall demonstrate 
elsewhere, the ideas underpinning the proof can be further developed, and we defer to this 
future occasion a detailed account of the potential of the method.\par

Our new major arc mean value estimates provide a versatile pruning device. As we shall 
see in Section 5, the bounds provided by Lemma \ref{lemma2.3} are of strength sufficient 
to establish a version of Theorem \ref{thm1.1} with an inflated value of $c$. We enhance 
the power of the new method with our second innovation, a novel large values technique. 
Drawing inspiration from an argument that occurs {\em en passant} in the first author's 
work on a certain quaternary additive problem \cite{B-Kac}, we explore the consequences 
of the stipulation that a Weyl sum is large, but not very large, through a comparison of 
various moments. This new method may be viewed as a pruning device for the minor arcs, 
and it transpires that in certain cases, it is possible to improve Weyl type bounds in mean 
over sets that we expect to be small. In applications of the circle method, this works just as 
if the Weyl bound would be better than currently known. Since the method ultimately rests 
on mean values, it also cooperates with the new major arc means, and then also helps with 
the more classical aspects of pruning on major arcs. The technical aspects of our new 
devices will be explained in the course of the argument, once the notational apparatus has 
been introduced. We refer the reader to the final part of Section 5, and the proof of Lemma 
\ref{lemma5.1} below, for details.\par

This paper is organised as follows. We begin with a discussion of the major arc mean values 
in Section 2. In Sections 3 and 4 we evaluate the major arcs in a circle method approach to 
the counting function $r_{k,s}$. This is largely standard. Then, in Section 5, we highlight the 
potential of major arc mean values for pruning. This section ends with the statement of 
Lemma \ref{lemma5.1} that in turn is dependent on the new large values technique. In 
Section 6 we present this as a pruning device on minor arcs, and in Section 7 in a catalytic 
role to enhance classical pruning. Having made the necessary preparations, the proof of our 
main theorems is presented in Section 8.

\section{Smooth Weyl sums and their means} In this section we discuss certain mean value 
estimates for smooth Weyl sums. Our discussion prominently features a certain 
transcendental function, which we now define. The function 
$\ome\colon (0,\infty)\rightarrow (0,\mathrm e)$, defined by 
$t\mapsto {\mathrm e}^{1-t}$, is a strictly decreasing bijection, while the function 
$\Ome\colon (0,1)\rightarrow (0,\mathrm e)$, defined by $u\mapsto u\,{\mathrm e}^u$, is 
a strictly increasing bijection. It follows that the equation 
$\Eta\, {\mathrm e}^\Eta = {\mathrm e}^{1-t}$ 
defines a smooth, strictly decreasing bijective function $\Eta: (0,\infty)\to (0,1)$. We then 
have
\begin{equation}\label{2.1}
\Eta(t)+\log \Eta(t)=1-t,
\end{equation}
and we may differentiate to infer that the relation
\begin{equation}\label{2.2}
\Eta'(t)=-\Eta(t)/(1+\Eta(t)) 
\end{equation}
holds for all $t>0$. Later we require the following simple inequality.

\begin{lemma}\label{lemma2.1}
If $1\le t\le 3$, then $\Eta(t)> 1/(4t-1)$.
\end{lemma} 

\begin{proof} Define $u(t)$ by means of the equation $\Eta(u(t)) = 1/(4t-1)$. We show 
that $u(t)>t$ for $1\le t\le 3$. Since $\Eta$ decreases, the desired inequality follows.\par 

By \eqref{2.1}, we have
\[
\frac{1}{4t-1}-\log (4t-1)=1-u(t),
\]
and hence
\[
\frac{\mathrm d}{\mathrm d t}(u(t)-t)=\frac{4}{4t-1}+\frac{4}{(4t-1)^2} -1.
\]
This function decreases on $[1,3]$ and is positive at $t=1$ but negative at $t=3$. Thus, the 
minimum of $u(t)-t$ on the interval $[1,3]$ occurs at $t=1$ or $t=3$. However, it is readily 
checked that $u(1)>1$ and $u(3)>3$, whence $u(t)-t$ is positive throughout $[1,3]$.
\end{proof} 

When $1\le R\le P$, let $\mathscr A(P,R)$ denote the set of integers $n\in [1,P]$, all of 
whose prime divisors are at most $R$. Given an integer $k\ge 2$, let
\begin{equation}\label{2.3}
f(\alpha;P,R)=\sum_{x\in \mathscr A(P,R)}e(\alpha x^k),
\end{equation}
where, as usual, we write $e(z)$ to denote ${\mathrm e}^{2\pi iz}$. In this paper, we 
refer to the number $\Delta_t$ as an {\it admissible exponent} for the positive real number 
$t$ if, for any fixed positive real number $\eps$ there exists a positive real number $\eta$ 
such that, whenever $1\le R\le P^\eta$, one has
\[
\int_0^1|f(\alpha;P,R)|^t\,\mathrm d\alpha \ll P^{t-k+\Delta_t+\varepsilon}.
\]
For large values of $k$ the smallest known admissible exponents are due to Wooley 
\cite[Theorem 3.2]{Woo1992}. We reproduce a simplified, slightly weaker version of this 
conclusion, which appears in \cite[Theorem 2.1]{W1}, in the following lemma. 

\begin{lemma}\label{lemma2.2}
Let $k\ge 3$ be given. Then, whenever $t$ is an even natural number, the exponent 
$k\Eta(t/k)$ is admissible.
\end{lemma}

Our next lemma is our development of a related estimate of Liu and Zhao \cite{LZ}. As 
pointed out in the introduction, it is indispensible in our approach to Theorem \ref{thm1.1}. 
From now on we consider $k\ge 3$ as fixed. Let $Q$ be a real number with 
$1\le Q\le \frac12 P^{k/2}$, and let $\mathfrak M=\mathfrak M(Q)$ denote the union of 
the intervals
\[
\mathfrak M(q,a;Q)=\{\alpha\in[0,1]: |q\alpha -a|\le QP^{-k}\}
\]
with $0\le a\le q\le Q$ and $(a,q)=1$. Note that the intervals constituting this union are 
disjoint. Our goal is to estimate the mean value $V=V_t(P,R,Q)$ defined by
\begin{equation}\label{2.5}
V_t(P,R,Q)=\int_{\mathfrak M(Q)}|f(\alpha;P,R)|^t\,\mathrm d\alpha. 
\end{equation}

\begin{lemma}\label{lemma2.3} Let $k\ge 3$ be given. Suppose that $t\ge k+1$ is a real 
number and let $\Delta_t$ be an admissible exponent for $t$. Then for each $\eps>0$, 
there exists a positive number $\eta$ with the property that whenever 
$1\le R\le P^\eta$ and $1\le Q\le \frac12 P^{k/2}$, one has the uniform bound
\[
V_t(P,R,Q)\ll P^{t-k+\varepsilon}Q^{2\Delta_t /k}.
\]
\end{lemma}

\begin{proof} In the initial phase of the argument to follow, fix a real number $\delta$ with 
$0<\delta<10^{-10k}$. We first dispose of the case where $Q$ is very small. The measure 
of $\mathfrak M(Q)$ is $O(Q^2P^{-k})$, and therefore, the trivial bound 
$|f(\alpha;P,R)|\le P$ implies that $V\ll Q^2 P^{t-k}$. In particular, whenever 
$1\le Q\le P^{k\delta/2}$, one has
\begin{equation}\label{2.6}
V_t(P,R,Q) \ll P^{t-k+k\delta}.
\end{equation}
Next, since $\mathfrak M(Q)\subset [0,1]$, the definition of an admissible exponent shows 
that for some number $\eta_0$ with $0<\eta_0<\delta/2$, when $1\le R\le P^{\eta_0}$, 
one has
\begin{equation}\label{2.7}
 V_t(P,R,Q) \ll P^{t-k+\Delta_t+\delta}. 
\end{equation}  

We now launch the main argument, working under the assumption that
\begin{equation}\label{2.8}
P^{k\delta}\le Q^2\le P^{k(1-\delta)}. 
\end{equation}
With the parameter $M$ at our disposal, we suppose that $2\le R\le M\le P$. Throughout, 
we reserve the letters $p$ and $\pi$ to denote prime numbers. Then, for each $p\le R$, we 
define the modified set of smooth numbers
\[
\mathscr B(M,p,R)=\{ m\in \calA(Mp,R):\text{$m>M$, $p|m$, and $\pi|m$ implies that 
$\pi\ge p$}\}.
\]
By \cite[Lemma 10.1]{V89}, there is a bijection between the numbers 
$x\in{\mathscr A}(P,R)$ with $x>M$, and triples $(p,m,y)$ with
\[
p\le R,\quad m\in{\mathscr B}(M,p,R),\quad y\in{\mathscr A}(P/m,p),
\]
in which one has $x=my$. Applied to the exponential sum $f$ defined in \eqref{2.3} this 
provides us with the decomposition
\[
f(\alpha;P,R)=\sum_{p,m}f(\alpha m^k;P/m,p)+f(\alpha;M,R),
\]
where, both here and in the sequel, the summation over $p,m$ is intended as shorthand for 
one over $p\le R$ and $m\in{\mathscr B}(M,p,R)$. We write $f(\alpha)$ for $f(\alpha;P,R)$ 
and $h_{p,m}(\gamma)$ for $f(\gamma;P/m,p)$. In addition, we denote by 
$\mathfrak M_q$ the union of the arcs $\mathfrak M(q,a;Q)$ with $0\le a\le q$ and 
$(a,q)=1$. Then, for $\alpha \in \mathfrak M_q$, we sort the sum over 
$m\in {\mathscr B}(M,p,R)$ according to the value of $(q,m^k)$. Thus
\begin{equation}\label{2.9}
|f(\alpha)|\le \sum_{d\mid q}\sum_{\substack{p,m\\ (q,m^k)=d}}|h_{p,m}(\alpha m^k)|
+M.
\end{equation}
\par

Taking the $t$-th power of \eqref{2.9}, H\"older's inequality combines with the familiar 
divisor function estimate to give
\begin{equation}\label{2.10} 
|f(\alpha)|^t\ll q^\delta \sum_{d|q}\biggl(\sum_{\substack{p,m\\ (q,m^k)=d}} 
|h_{p,m}(\alpha m^k)|\biggr)^t+M^t.
\end{equation}
The integer $d$ has a unique factorisation $d=d_1d_2^2\cdots d_k^k$ with 
$d_1d_2\cdots d_{k-1}$ square-free. In this notation, we write $d_0=d_1d_2\cdots d_k$. 
Thus $d|m^k$ if and only if $d_0|m$. By H\"older's inequality again, 
\begin{equation}\label{2.11}
\biggl(\sum_{\substack{p,m\\ (q,m^k)=d}}|h_{p,m}(\alpha m^k)|\biggr)^t 
\le \biggl(\sum_{\substack{p,m\\ d_0|m}}1\biggr)^{t-1} 
\sum_{\substack{p,m\\ (q,m^k)=d}}|h_{p,m}(\alpha m^k)|^t.
\end{equation}
We integrate \eqref{2.10} over $\mathfrak M_q$. For this we require the mean value
\[
J(q,p,m)=\int_{\mathfrak M_q}|h_{p,m}(\alpha m^k)|^t\,\mathrm d \alpha.
\]
The first sum on the right hand side of \eqref{2.11} is no larger than $MR^2/d_0$. 
Moreover, the measure of $\mathfrak M_q$ does not exceed $2QP^{-k}$. Then 
\eqref{2.10} yields
\begin{equation}\label{2.12}
\int_{\mathfrak M_q}  |f(\alpha)|^t\,\mathrm d \alpha \ll q^\delta \sum_{d\mid q}
\Big(\frac{MR^2}{d_0}\Big)^{t-1} \sum_{\substack{p,m\\ (q,m^k)=d}} J(q,p,m)+
M^tQP^{-k}. 
\end{equation}

\par Temporarily, we consider $q,p,m$ as fixed and abbreviate $h_{p,m}$ to $h$. Also, 
when $1\le Z\le P$, we introduce notation to better handle the intervals $I_q(Z)=I_q(Z;Q)$ 
of interest, writing
\[
I_q(Z;Q)=[-Q/(qZ^k),Q/(qZ^k)].
\] 
Equipped with this notation, if we unfold the definition of $\mathfrak M_q$ and apply a 
change of variables, we see that the mean value $J(q,p,m)$ is equal to
\[
\int_{I_q(P)}\sum_{\substack{a=1\\ (a,q)=1}}^q 
\Big|h\Bigl(\Bigl(\frac{a}{q}+ \beta\Bigr) m^k\Bigr)\Big|^t\,\mathrm d \beta 
= \frac1{m^{k}} \int_{I_q(P/m)} \sum_{\substack{a=1\\ (a,q)=1}}^q 
\Big|h\Bigl(\frac{am^k}{q}+ \gamma\Bigr) \Big|^t\,\mathrm d \gamma.
\]
Within \eqref{2.12} we need the above relation only when $(q,m^k) =d$. In the latter 
circumstances one has $(m^k/d, q/d)=1$, and hence
\begin{align*}
J(q,p,m)&\le \frac{1}{m^k}\int_{I_q(P/m)}d\sum_{\substack{a=1\\ (a,q/d)=1}}^{q/d}
\Big|h\Bigl( \frac{am^k/d}{q/d}+ \gamma\Bigr) \Big|^t\,\mathrm d\gamma \\
&=\frac{d}{m^{k}}  \sum_{\substack{b=1\\ (b,q/d)=1}}^{q/d}\int_{I_q(P/m)}
\Big|h\Bigl(\frac{b}{q/d}+ \gamma\Bigr) \Big|^t\,\mathrm d \gamma.
\end{align*}
We apply this formula within \eqref{2.12} and sum over $q\le Q$ to obtain the bound
\begin{equation}\label{2.13}
\int_{\mathfrak M(Q)}|f(\alpha)|^t\,\mathrm d\alpha \ll M^tQ^2P^{-k}+
Q^\delta (MR^2)^{t-1}\Xi, 
\end{equation}
where
\[
\Xi=\sum_{q\le Q}\sum_{d|q}\sum_{\substack{p,m\\ (q,m^k)=d}}
\frac{d}{d_0^{t-1}m^k}\sum_{\substack{b=1\\ (b,q/d)=1}}^{q/d}
\int_{I_q(P/m)}\Big| h\Bigl( \frac{b}{q/d}+ \gamma\Bigr) \Big|^t\,\mathrm d \gamma.
\]

\par The quantity $\Xi$ may be bounded by first writing $q=dr$, and then replacing the 
condition $(q,m^k)=d$ by the weaker constraint $d_0|m$. Thus, we find that
\begin{equation}\label{2.14}
\Xi \le \sum_{dr\le Q} \sum_{\substack{p,m\\ d_0|m}}\frac{d}{d_0^{t-1} m^k} 
\sum_{\substack{b=1\\ (b,r)=1}}^r 
\int_{I_{dr}(P/m)}\Big| h\Bigl( \frac{b}{r}+\gamma\Bigr) \Big|^t\,\mathrm d \gamma.
\end{equation}
We  choose $M$ via the relation $(MR)^k= P^k/(3Q^2)$, whence by \eqref{2.8} one has 
\begin{equation}\label{2.15}
P^\delta \ll  MR \ll P^{1-\delta}.
\end{equation}
In particular, this choice for $M$ is admissible and for all pairs $p,m$ occurring in the 
second sum of \eqref{2.14}, we have $Q^2(P/m)^{-k}<\frac12$. It is immediate from this 
inequality that the intervals $I_{dr}(P/m)+b/r$, with $1\le b\le r\le Q/d$ and $(b,r)=1$, are 
pairwise disjoint, and that all of these sets are contained in the unit interval
\[
\left[\frac{Qm^k}{drP^k}, 1 +\frac{Qm^k}{drP^k}\right].
\]
We therefore conclude that
\[
\sum_{r\le Q/d}\sum_{\substack{b=1\\ (b,r)=1}}^r \int_{I_{dr}(P/m)} 
\Big|h\Bigl(\frac{b}{r}+ \gamma\Bigr) \Big|^t\,\mathrm d \gamma 
\le \int_0^1 |h_{p,m}(\alpha)|^t\,\mathrm d \alpha.
\]
But from \eqref{2.15} we have $P/m\ge P/(MR)\gg P^{\delta}$. Since $p\le R$ here, it 
follows from the definition of an admissible exponent that there is a number $\eta>0$ such 
that, uniformly in $2\le R\le P^\eta$ and all $p,m$ in their ranges of summation,
\[
\int_0^1 |h_{p,m}(\alpha)|^t\,\mathrm d \alpha \ll (P/m)^{t-k+\Delta_t+\delta}.
\]
The exponent on the right hand side is positive, so we infer from \eqref{2.14} that
\begin{equation}\label{2.16}
\Xi \ll \Bigl( \frac{P}{M}\Bigr)^{t-k+\Delta_t+\delta}\sum_{d\le Q}
\sum_{\substack{p,m\\ d_0|m}}\frac{d}{d_0^{t-1}m^k}.
\end{equation}

\par One has
\[
\sum_{p\le R}\sum_{\substack{m\in \mathscr B(M,p,R)\\ d_0|m}}\frac{1}{m^k}\le 
\sum_{p\le R}\sum_{M/d_0<u\le Mp/d_0}\frac{1}{(d_0u)^k}\ll \frac{R}{d_0M^{k-1}}.
\]
Moreover, when $t\ge k+1$, one has the simple bound
\[
\sum_{d\le Q}\frac{d}{d^{\,t}_0}\ll \sum_{d_1d_2\cdots d_k\le Q}
\frac{1}{d_1d_2\cdots d_k}\ll (\log (2Q))^k.
\]
Recall that our choice for $Q$ satisfies $(P/M)^k=3Q^2R^k$ and $Q\le P^{k/2}$. Then on 
noting that $MR\gg P^\delta$, it follows from \eqref{2.16} that
\begin{align*}
\Xi &\ll (P/M)^{t-k+\delta}Q^{2\Delta_t/k}R^{1+\Delta_t}M^{1-k}(\log (2Q))^k\\
&\ll P^{t-k+\delta}Q^{2\Delta_t/k}M^{1-t}R^{2+\Delta_t}.
\end{align*}
We may always assume that $\Delta_t\le k< t$, and thus we infer from \eqref{2.5} and 
\eqref{2.13}, together with the definition of $M$, the upper bound
\begin{equation}\label{2.17}
V_t(P,R,Q) \ll P^{t-k+\delta k} Q^{2\Delta_t/k} R^{3t}. 
\end{equation}

We are ready to deduce Lemma \ref{lemma2.3}. Fix $\varepsilon>0$. We may assume of 
course that $\varepsilon<10^{-10k}$. Since we may diminish the number $\eta$ implicit in 
the proof of \eqref{2.17}, we may arrange that $2\delta k+3t\eta<\varepsilon$. 
Consequently, provided that the constraint \eqref{2.8} is satisfied, the conclusion of Lemma 
\ref{lemma2.3} follows from \eqref{2.17}. Should the condition \eqref{2.8} not be 
satisfied, then we find ourselves in one of two situations. In the first situation 
$Q^2>P^{k(1-\delta)}$, in which case \eqref{2.7} yields
\[
P^{t-k+2k\delta}Q^{2\Delta_t/k}\gg P^{t-k+\Delta_t(1-\delta)+2k\delta}\gg V_t(P,Q,R).
\]
In the alternative situation one has  $Q^2<P^{k\delta}$, in which case \eqref{2.6} yields
\[
V_t(P,Q,R)\ll P^{t-k+k\delta}\ll P^{t-k+k\delta}Q^{2\Delta_t/k}.
\]
In both scenarios, the conclusion of the lemma follows because $2k\delta<\varepsilon$.   
\end{proof}

\section{The circle method}
Our next task is to set up the environment for the proofs of the theorems. The exponent 
$k\ge 3$ is still fixed, and we initially impose the condition $s\ge 1$. We gradually import 
more conditions on $s$ and the smoothness parameter as the argument progresses. Our 
leading parameter is the number $n$ in \eqref{1.1}, and we take
\begin{equation}\label{3.1}
P= n^{1/k}.
\end{equation}
We adumbrate $f(\alpha;P,R)$ to $f(\alpha)$ and introduce the sum
\[
g(\alpha)=\sum_{p \le n}e(\alpha p) \log p.
\]
Whenever $\mathfrak A\subset [0,1] $ is measurable, we write
\begin{equation}\label{3.2}
I(n,\mathfrak A)=\int_{\mathfrak A}g(\alpha)f(\alpha)^s e(-\alpha n)\,\mathrm d \alpha
\end{equation}
and abbreviate $I(n,[0,1])$ to $I(n)$. By orthogonality, the integral counts certain solutions 
of \eqref{1.1} with weight $\log p$. Consequently,
\begin{equation}\label{3.3}
r_{k,s}(n)\log n \ge I(n).
\end{equation} 

The arguments in this section are independent of the theory of admissible exponents. We 
therefore choose $R=P^\delta$, where $0<\delta\le 1$ remains at our disposal. It is 
convenient also to write $\mathscr Q=(\log n)^{1/99}$. We begin by extracting a lower 
bound from the core major arcs that we define as the union of the intervals
\[
\mathfrak N(q,a)=\{\alpha \in [0,1]: |\alpha -a/q|\le \mathscr Q n^{-1}\},
\]
with $0\le a\le q\le \mathscr Q$ and $(a,q)=1$. The intervals in this union are again 
disjoint. Let $\alpha$ be in one of these intervals, say $\mathfrak N(q,a)$ with $(a,q)=1$. 
On recalling \eqref{3.1}, arguments that are by now standard in the theory of smooth Weyl 
sums (see \cite[Lemma 5.4]{V89}) show in this scenario that there is a positive number 
$\rho=\rho(\delta)$ such that
\[
f(\alpha)=\rho q^{-1} S(q,a)v_k(\alpha-a/q)+O(n^{1/k}(\log n)^{-1/4}),
\]
wherein
\[
S(q,a)=\sum_{x=1}^q e(ax^k/q)\quad \text{and}\quad 
v_k(\beta)=\frac{1}{k}\sum_{m\le n} m^{-1+1/k}e(\beta m).
\]
We take the $s$-th power and combine the result with Lemma 3.1 of Vaughan \cite{hlm}. 
After integrating over $\mathfrak N$ we then routinely obtain the asymptotic relation
\begin{equation}\label{3.4}
I(n,\mathfrak N) = \rho^s\mathfrak J(n,\mathscr Q) \mathfrak S(n,\mathscr Q)+
O(n^{s/k} (\log n)^{-1/5}),
\end{equation}
where, for $1\le X\le n/2$, we write
\[
\mathfrak J(n,X)=\int_{-X/n}^{X/n} v_1(\beta)v_k(\beta)^s e(-\beta n)\,\mathrm d\beta \]
and 
\begin{equation}\label{3.5}
\mathfrak S(n,X)=\sum_{q\le X}\frac{\mu(q)}{q^s\phi(q)}
\sum_{\substack{a=1\\ (a,q)=1}}^q S(q,a)^s e(-an/q).
\end{equation}

\par By orthogonality, it is immediate that $\mathfrak J(n,n/2)\gg n^{s/k}$ (see, for 
example, the proof of \cite[Theorem 2.3]{hlm}). Moreover, we find via 
\cite[Lemma 2.8]{hlm} that
\[
\mathfrak J(n,X)-\mathfrak J(n,n/2) \ll n^{s/k}X^{-s/k}.
\]
Thus we see that
\begin{equation}\label{3.6}
\mathfrak J(n,\mathscr Q)\gg n^{s/k}.
\end{equation}
Meanwhile, the sum
\[
S_n(q)=q^{-s}\sum_{\substack{a=1\\ (a,q)=1}}^q S(q,a)^s e(-an/q)
\]
is multiplicative (see \cite[Lemma 2.11]{hlm}), and the presence of the factor $\mu(q)$ in 
\eqref{3.5} reduces our task to bounding $S(q)$ when $q$ is a prime. In these 
circumstances, whenever $s\ge 3$, we see from \cite[Lemma 4.3]{hlm} that the bound 
$S_n(q)\ll q^{-1/2}$ holds uniformly in $n$. This in turn implies that the series 
\[
\mathfrak S(n)=\lim_{X\to\infty}\mathfrak S(n,X)
\]
converges absolutely, and that
\begin{equation}\label{3.7}
\mathfrak S(n)-\mathfrak S(n,X)\ll X^{\varepsilon-1/2}.
\end{equation}
This bound again holds uniformly in $n$.

\par Next, we transform $\mathfrak S(n)$ into an Euler product $\prod_p\chi_p(n)$, where
\[
\chi_p(n)=1-(p-1)^{-1}S_n(p)=1+O(p^{-3/2}),
\]
the last relation holding uniformly in $n$. Since $-1$ is equal to the value of the Ramanujan 
sum $c_p(a)$ for $1\le a\le p-1$, we find that
\begin{align*}
\chi_p(n)&=1+\frac{1}{p^{s}(p-1)}\sum_{a=1}^{p-1}c_p(a)S(p,a)^se(-an/p)\\
&=\frac{1}{p^{s}(p-1)}\sum_{a=1}^{p} c_p(a)S(p,a)^s e(-an/p).
\end{align*}
Thus, by orthogonality, we have $\chi_p(n)=p^{1-s}(p-1)^{-1}M_p(n)$, where $M_p(n)$ is 
the number of incongruent solutions of the congruence
\[
b+x_1^k+\ldots+x_s^k\equiv n \mmod{p},
\]
with $(b,p)=1$. It is immediate that the lower bound $M_p(n)\ge 1$ holds for all $n$. 
Hence all factors $\chi_p(n)$ are real with $\chi_p(n)\ge p^{-s}$. From this discussion we 
also see that there is a number $p_0$ such that the superior lower bound 
$\chi_p(n)\ge 1-p^{-5/4}$ holds for all $n$ whenever $p\ge p_0$. We therefore conclude 
routinely that $\mathfrak S(n) \gg 1$. If we combine this lower bound with \eqref{3.4}, 
\eqref{3.6} and \eqref{3.7}, we may conclude as follows.

\begin{lemma}\label{lemma3.1}
Let $k\ge 3$ and $s\ge 3$. Suppose that $0<\delta\le 1$, and put $R=n^{\delta/k}$. Then 
one has $I(n,\mathfrak N)\gg n^{s/k}$.
\end{lemma}

A slightly weaker version of this lemma remains valid even when $s=2$. In fact, in this 
case, the bound $S_n(p)\ll p^{-3/2}$ remains valid for all primes $p\nmid n$. This can been 
seen by working along the lines of \cite[Lemma 4.7]{hlm}. Further, with more care one can 
show that $\chi_p(n)\ge 1-k/p$ for primes $p|n$ with $p>k^2$. As this is not used later, we 
leave the details required to justify these claims to the reader. These bounds show that 
when $s=2$, one has
\[
\mathfrak S(n)\gg \prod_{\substack{p|n\\ p>k^2}}(1- kp^{-1}) \gg (\log\log n)^{-k},
\]
whence $I(n,\mathfrak N)\gg n^{s/k} (\log\log n)^{-k}$. This supports our claim in the 
introduction that all large integers should be the sum of a prime and two $k$-th powers.  

\section{Pruning near the root}
Our goal in this section is to enlarge the major arcs to the set 
$\mathfrak L=\mathfrak M(P^{1/5})$. The choice of height $P^{1/5}$ here is arbitrary, for 
any small power of $P$ would suffice. As is often the case with competitive applications of 
smooth Weyl sums, certain estimates are diluted by unwanted factors $P^\varepsilon$. In 
such circumstances, pruning to the root needs a separate argument that we now present.

\par We begin by introducing some notation. Let $a\in \mathbb Z$ and $q\in \mathbb N$ 
satisfy $0\le a\le q\le \frac12 \sqrt n$ and $(a,q)=1$. Then, the intervals 
$\mathfrak M(q,a; \frac12 \sqrt n)$ are disjoint, and for 
$\alpha\in \mathfrak M(q,a;\frac12 \sqrt n)$ we put
\[
\Upsilon(\alpha)=(q+n|q\alpha-a|)^{-1}.
\]
Meanwhile, for $\alpha\in [0,1]\setminus \mathfrak M(\frac12 \sqrt n)$ we put 
$\Upsilon(\alpha)=0$. This defines a function $\Upsilon:[0,1]\to [0,1]$. In what follows, in 
the interest of brevity we put
\[
L=\log n .
\]

\begin{lemma}\label{lemma4.1}
Suppose that $2\le R\le P^{1/7}$ and $\varepsilon>0$. Then, uniformly for 
$\alpha\in\mathfrak L$ one has
$$ f(\alpha) \ll PL^3 \Upsilon(\alpha)^{1/(2k) -\varepsilon} . $$
Moreover, when $B>0$ and $\alpha\in\mathfrak M(L^B)$, one has 
$$ f(\alpha) \ll P\Upsilon(\alpha)^{1/k-\varepsilon}$$
\end{lemma}

\begin{proof} The first bound  follows by applying \cite[Lemma 7.2]{VWrefine} with 
$M=P^{2/3}$. The second bound, on the other hand, is a consequence of 
\cite[Lemma 8.5]{VWrefine}.
\end{proof}

\begin{lemma}\label{lemma4.2} Let $\alpha\in[0,1]$. Then
\begin{equation}\label{4.1}
g(\alpha)\ll \left( n\Upsilon(\alpha)^{1/2} +n^{4/5}\right) L^4.
\end{equation}
Furthermore, should GRH be true, then 
\begin{equation}\label{4.2}
g(\alpha) \ll (n\Upsilon(\alpha)^{1/2} +n^{3/4})L^2.
\end{equation}
\end{lemma}

\begin{proof} Both bounds \eqref{4.1} and \eqref{4.2} are certainly familiar, but perhaps 
not in this form, and in particular the latter is not easily found in the literature. We therefore 
provide details for the upper bound \eqref{4.2}. Let
\[
g(\alpha,\nu)=\sum_{p\le \nu}e(\alpha p)\log p .
\]
If we suppose that GRH holds, and $a\in\mathbb Z$ and $q\in\mathbb N$ are coprime, 
then one finds from \cite[Lemma 2]{BP}, for example, that
\begin{equation}\label{4.3}
g(a/q,\nu)\ll (\nu \phi(q)^{-1}+\sqrt{\nu q})(\log \nu)^2.
\end{equation}
Note that $g(\alpha)=g(\alpha,n)$. By partial summation, for all $\beta\in\mathbb R$, one 
has
\begin{align}
|g(\beta+a/q)|&\le |g(a/q)|+2\pi |\beta| \int_2^n |g(a/q,\nu)|\,\mathrm d\nu \notag\\
& \ll (n\phi(q)^{-1}+\sqrt{nq}) (1+n|\beta|) L^2.\label{4.4} 
\end{align}

In order to obtain the bound \eqref{4.2}, first suppose that 
$\alpha\in[0,1]\setminus \mathfrak M(\frac12 \sqrt n)$. Apply Dirichlet's theorem on 
Diophantine approximation to find integers $b$ and $r$ with $1\le r\le 2\sqrt n$, $(b,r)=1$ 
and $|r\alpha-b|\le 1/(2\sqrt n)$. Then the hypothesis 
$\alpha\not\in \mathfrak M(\frac12 \sqrt n)$ implies that $r>\frac12 \sqrt n$, and thus we 
infer from \eqref{4.4} that
\[
g(\alpha)\ll \Big( \frac{n}{\phi(r)}+\sqrt{nr}\Big)\Big(1+\frac{\sqrt n}{r}\Big) L^2 \ll 
n^{3/4}L^2,
\]
as required in \eqref{4.2}.\par

Next suppose that $\alpha\in\mathfrak M(\frac12\sqrt n)$. For 
$\frac12\le Y \le \frac14 \sqrt n$, we write
\begin{equation}\label{4.5}
\mathfrak P(Y)=\mathfrak M(2Y)\setminus \mathfrak M(Y).
\end{equation}
Then since $\mathfrak M(\frac12\sqrt n)$ is the union of the sets $\mathfrak P(Y)$, there 
exists a choice for $Y$ in this range with $\alpha\in \mathfrak P(Y)$. Note that for $\alpha$ 
in the latter set, the bound \eqref{4.2} asserts that $g(\alpha)\ll nY^{-1/2}L^2$. To prove 
this bound, we again apply Dirichlet's theorem to find integers $b$ and $r$ with 
$1\le r\le  n/Y$, $(b,r)=1$ and $|r\alpha-b|\le Y/n$. We may suppose in this instance that 
$\alpha\not\in \mathfrak M(Y)$, and hence that $r>Y$, and thus \eqref{4.4} yields the 
desired bound
\[
g(\alpha)\ll \Big( \frac{n}{\phi(r)} + nY^{-1/2}\Big)\Big(1+\frac{Y}{r}\Big) L^2\ll 
nY^{-1/2}L^2.
\]

\par The proof of the bound \eqref{4.2} is now complete. The proof of \eqref{4.1} follows 
in the same way, save that the initial estimate \eqref{4.3} has to be replaced by 
Vinogradov's unconditional bound \cite[Theorem 3.1]{hlm}.
\end{proof}

We are now well prepared for the pruning.

\begin{lemma}\label{lemma4.3}
Suppose that $s\ge k+3$ and $2\le R\le  P^{1/7}$. Then
\[
I(n,{\mathfrak L}\setminus \mathfrak N) \ll n^{s/k}(\log n)^{-1/(50k)}.
\]
\end{lemma}

\begin{proof} We proceed in two steps. We first note that an enhanced version (see 
\cite[Lemma 11.1]{PW2014}) of the first author's pruning technique 
\cite[Lemma 1]{B89} yields
\[
\int_{\mathfrak L}\Upsilon(\alpha)^{1+1/(6k)}|f(\alpha)|^2\,\mathrm d\alpha \ll 
P^2n^{-1}.
\] 
By \eqref{4.1} and Lemma \ref{lemma4.1}, we have
\[
g(\alpha)f(\alpha)^{k+1}\ll P^{k+1}n\Upsilon(\alpha)^{1+1/(3k)}L^{7+3k}.
\]
Take $B=6k(8+3k)$. Then, for $\alpha\in\mathfrak L\setminus \mathfrak M(L^B)$, one has 
$\Ups(\alpha)\ll L^{-B}$, whence
\begin{equation}\label{4.6}
I(n,\mathfrak L\setminus \mathfrak M(L^B))\ll L^{-1}P^{s-2}n 
\int_{\mathfrak L} \Upsilon(\alpha)^{1+1/(6k)}|f(\alpha)|^2\,\mathrm d\alpha 
\ll P^sL^{-1}.
\end{equation}
For $\alpha\in\mathfrak M(L^B)$, one finds from \cite[Lemma 3.1]{hlm} and familiar 
estimates that
\[
g(\alpha) \ll n\Upsilon(\alpha) \log L.
\]
Observe that when $\alpha\in [0,1]\setminus \mathfrak N$, one has 
$\Upsilon(\alpha)\ll L^{-1/99}$. Consequently, by making use of the second estimate of 
Lemma \ref{lemma4.1}, we discern that whenever 
$\alpha \in \mathfrak M(L^B)\setminus \mathfrak N$, then
\[
g(\alpha)f(\alpha)^s\ll P^s n \Upsilon(\alpha)^{2+1/k}L^{-1/(50k)}.
\]
We may therefore integrate routinely to conclude that
\begin{equation}\label{4.7}
I(n,\mathfrak M(L^B)\setminus \mathfrak N) \ll P^s L^{-1/(50k)}.
\end{equation}
The proof of the lemma is completed by collecting together \eqref{4.6} and \eqref{4.7}.
\end{proof} 

\section{Interlude}
Thus far our evaluation of $r_{k,s}(n)$ has followed a well trodden path. Before embarking 
on the original aspects of our treatment, we pause to outline a simple argument leading 
to weaker versions of Theorems \ref{thm1.1} and \ref{thm1.3}, with inflated numerical 
values for the constants $c$ and $c'$. This approach rests on Lemma \ref{lemma2.3} 
alone, and is therefore very flexible. It applies also in other contexts.\par

From now on, in the remainder of this paper, we apply the following conventions concerning 
the symbols $\varepsilon$, $\eta$ and $R$. If a statement involves the letter $\varepsilon$, 
then it it is asserted that the statement holds for any positive real number assigned to 
$\varepsilon$. Implicit constants stemming from Vinogradov or Landau symbols may depend 
on $\varepsilon$. If a statement also involves the letter $R$, either implicitly or explicitly, 
then it is asserted that for any $\varepsilon>0$ there is a number $\eta>0$ such that the 
statement holds uniformly for $2\le R\le P^\eta$. This will be imported to our arguments 
through applications of Lemmata \ref{lemma2.2} and \ref{lemma2.3} only. We shall call 
upon these lemmata only finitely often, and we may therefore pass to the smallest of the 
numbers $\eta$ that arise in this way, and then have all estimates with the same positive 
$\eta$ in hand.\par

The Farey dissection that we shall use takes $\mathfrak K = \mathfrak M(n^{2/5})$ as the 
set of major arcs, and $\mathfrak k= [0,1]\setminus \mathfrak K$ as the complementary 
set of minor arcs. Let $c_1$ be the positive real number with $\Eta(c_1)=\frac15$. Note 
from \eqref{2.1} that
\[
c_1=\frac45+\log 5=2.409437\ldots.
\]
Let $s_1=s_1(k)$ be the smallest even integer with $s_1>c_1k$. Then, for $s\ge s_1$ we 
have $\Eta(s/k)<\frac15$. By Lemma \ref{lemma2.2} and the definition of an admissible 
exponent, one sees that for some positive number $\delta$, one has
\begin{equation}\label{5.1}
\int_0^1 |f(\alpha)|^s\,\mathrm d\alpha \ll P^sn^{-4/5-\delta}.
\end{equation}
Also, as a consequence of Lemma \ref{lemma2.3}, whenever $1\le Q\le \frac12 \sqrt n $, 
we have
\begin{equation}\label{5.2}
\int_{\mathfrak M(Q)} |f(\alpha)|^s\,\mathrm d\alpha 
\ll P^sn^{\varepsilon-1}Q^{2/5-\delta}.
\end{equation}
Note here that our conventions concerning the use of $\varepsilon$ and $R$ apply. In 
particular, we may choose $R=P^\eta$ and then the upper bounds \eqref{5.1} and 
\eqref{5.2} both hold provided that $\eta>0$ is sufficiently small. We fix this choice of $R$ 
now.\par
 
We begin by establishing a version of Theorem \ref{thm1.1}. Observe first that when 
$\alpha \in \mathfrak k$, then as a consequence of Dirichlet's approximation theorem, 
whenever $a\in \mathbb Z$ and $q\in \mathbb N$, one has $q+n|q\alpha -a|>n^{2/5}$. 
Hence, by Lemma \ref{lemma4.2}, we have $g(\alpha)\ll n^{4/5+\varepsilon}$. We thus 
deduce from \eqref{3.2} and \eqref{5.1} that
\begin{equation}\label{5.3}
I(n,\mathfrak k)\ll P^s n^{-\delta/2}.
\end{equation}
Next recall \eqref{4.5} and apply a dyadic dissection to cover the set 
$\mathfrak K\setminus\mathfrak L$ by $O(\log n)$ sets $\mathfrak P(Q)$, with 
$P^{1/5}\le Q\le \frac12n^{2/5}$. For $\alpha\in \mathfrak P(Q)$, we infer from Lemma 
\ref{lemma4.2} that $g(\alpha)\ll n^{1+\varepsilon}Q^{-1/2}$. Hence, by \eqref{3.2} and 
\eqref{5.2}, we discern that
\[
I(n,\mathfrak P(Q))\ll P^s Q^{-1/10}.
\]
Collecting together the contributions from these dyadic intervals, we see that
\begin{equation}\label{5.4}
I(n,\mathfrak K\setminus\mathfrak L)\ll P^s n^{-1/(60k)}.
\end{equation}
Thus, on recalling Lemmata \ref{lemma3.1} and \ref{lemma4.3}, we deduce from 
\eqref{5.3} and \eqref{5.4} that
\[
I(n)=I(n,\mathfrak k)+I(n,\mathfrak K\setminus\mathfrak L)+
I(n,\mathfrak L\setminus \mathfrak N)+I(n,\mathfrak N)\gg n^{s/k}.
\]
In view of \eqref{3.3}, we therefore have $r_{k,s}(n)\gg n^{s/k}(\log n)^{-1}$ whenever 
$s\ge c_1k+1$, thereby delivering a version of Theorem \ref{thm1.1} with $c_1$ in place 
of $c$.\par

If GRH is true we adjust the Farey dissection to one comprising the sets 
$\mathfrak K'=\mathfrak M(\frac12\sqrt n)$ and 
$\mathfrak k'=[0,1]\setminus \mathfrak K'$. Also, we cover the set 
$\mathfrak K'\setminus \mathfrak L$ by $O(\log n)$ sets $\mathfrak P(Q)$, with 
$P^{1/5}\le Q\le \tfrac{1}{4}\sqrt{n}$. In this scenario we take
\[
c'_1=\frac34 +\log 4=2.136294\ldots ,
\]
so that in view of \eqref{2.1}, one has $\Eta(c'_1)=\frac14$. Let $s'_1=s'_1(k)$ be the 
smallest even integer with $s'_1>c'_1 k$. The preceding argument may now be based on 
\eqref{4.2} rather than \eqref{4.1}, and then shows that 
$r_{k,s}(n)\gg n^{s/k}(\log n)^{-1}$ whenever $s\ge c'_1k+1$, yielding a version of 
Theorem \ref{thm1.3} with $c'_1$ in place of $c'$. It is perhaps of interest to note that 
even this conditional conclusion is improved on by Theorem \ref{thm1.1}.\par

Applications of the Hardy-Littlewood method to the problem at hand start with a minor arc 
analysis associated with the set $[0,1]\setminus \mathfrak M(Q)$, with a choice for $Q$ 
large enough that estimates of Weyl-type are applicable. Traditional pruning arguments 
attempt to reduce the size parameter $Q$ so that the associated complementary set of 
major arcs $\mathfrak M(Q)$ becomes workable. Ideally, this parameter $Q$ should not be 
very large, and indeed $Q=P^{1/5}$ in the set $\mathfrak L$ occurring in our argument 
above. We refer to $Q$ as the {\em height} of the major arcs, and a pruning argument 
that reduces $Q$ is referred to as {\em pruning by height}. There are several models for 
this kind of argument where one would typically combine pointwise bounds for some 
generating functions with mean values of others. One standard tool is \cite[Lemma 1]{B89}. 
It transpires that Lemma \ref{lemma2.3} is a particularly powerful technique for pruning by 
height because it is initially based on mean values alone.\par

In the next section we introduce another new pruning device. In its pure form it is a minor 
arc technique. We propose a dissection into level sets for several generating functions. If a 
generating function is very large, then inequalities of Weyl's type will tell us that we are on 
major arcs. We shall then explore via mean values the proposition that a generating 
function is large but not very large. In favourable circumstances one ends up with results 
that have, for the problem at hand, the same effect as an improvement on the Weyl 
bounds. We refer to this process as {\em pruning by size}. In our application, the new 
method performs so well that pure pruning by height on the major arcs would become the 
bottleneck for the argument. Hence, in the section following the next one, we describe an 
argument where pruning by height is enhanced by some of the features of pruning by size. 
It is interesting to note already at this stage that the success of pruning by size and by 
height ultimately depends on the same inequalities for certain admissible exponents. It is 
possible to announce the principal conclusion in a form that absorbs potential improvements 
on admissible exponents easily. This concerns the contribution from the set
\[
\mathfrak l = [0,1]\setminus \mathfrak L.
\]

\begin{lemma}\label{lemma5.1} Let $k\ge 3$. Suppose that $s$ and $t$ are real numbers 
with $0\le t\le s$, and let $\Delta_s$ and $\Delta_{s+t}$ be admissible exponents. Should 
GRH be true, set $\theta=4$, and otherwise put $\theta=5$. Then, provided that
\begin{equation}\label{5.5}
2\Delta_s <k\quad \text{and}\quad \frac{t}{s} + \frac{\theta\Delta_{s+t}}{k} < 1,
\end{equation}
there exists a positive number $\eta$ such that, uniformly for $2\le R\le P^\eta$, one has
\[
\int_{\mathfrak l} |g(\alpha) f(\alpha)^s|\,\mathrm d\alpha \ll P^s(\log n)^{-1}.
\]
\end{lemma}   

\section{Pruning by size}
We prove Lemma \ref{lemma5.1} in two steps. In this section we deal with the minor arcs 
$\mathfrak k$ and, subject to GRH, the alternative minor arcs $\mathfrak k'$. Throughout 
we suppose that the hypotheses in Lemma \ref{lemma5.1} are met, with $\theta$ 
depending on and also determining the case under consideration.\par

We begin by removing from the minor arcs the set
\[
\mathscr T=\{\alpha\in \mathfrak k:  |g(\alpha)|\le \sqrt n\},
\]
where $g(\alpha)$ is tiny. By \eqref{5.5} and the definition of an admissible exponent, there 
is a number $\delta>0$ with
\[
\int_0^1 |f(\alpha)|^s\,\mathrm d\alpha \ll P^s n^{-1/2-\delta},
\]
whence
\begin{equation}\label{6.1}
\int_{\mathscr T}|g(\alpha)f(\alpha)^s|\,\mathrm d\alpha \ll P^s n^{-\delta}.
\end{equation}

Next, let $U$ be a parameter with $1\le U\le \sqrt n$, and define the level set
\begin{equation}\label{6.2}
{\mathscr L}(U)=\{\alpha\in [0,1]: n/U\le |g(\alpha)|\le 2n/U\}. 
\end{equation}
In the interest of brevity we revive the use of $L=\log n$. Also, let $U_0=n^{1/\theta}$ 
and note that whenever $U\le U_0L^{-5}$, one must have $\alpha\in\mathfrak K$ (if 
$\theta=5$), or $\alpha\in\mathfrak K'$ (if $\theta=4$), at least for large $n$. This follows 
from Lemma \ref{lemma4.2}, which shows that in these respective cases one has
\[
\sup_{\alpha \in \mathfrak k}|g(\alpha)|\ll n^{4/5}L^4\quad \text{and}\quad 
\sup_{\alpha \in \mathfrak k'}|g(\alpha)|\ll n^{3/4}L^2.
\]
Consequently, we may cover the set $\mathfrak k\setminus \mathscr T$ by $O(L)$ sets 
${\mathscr L}(U)$ with
\begin{equation}\label{6.3}
n^{1/5}L^{-5}\le U\le \sqrt n. 
\end{equation}
Likewise, the set $\mathfrak k'$ is the union of $\mathscr T$ and $O(L)$ sets 
${\mathscr L}(U)$ with 
\begin{equation}\label{6.4}
n^{1/4}L^{-3}\le U\le \sqrt n . 
\end{equation}
Hence, we find via \eqref{6.1} that there is a number $U$ satisfying \eqref{6.3} for which
\begin{equation}\label{6.5}
\int_{\mathfrak k}|g(\alpha)f(\alpha)^s|\,\mathrm d\alpha \ll P^sn^{-\delta}
+L\int_{\mathscr L(U)} |g(\alpha)f(\alpha)^s|\,\mathrm d\alpha. 
\end{equation}
This bound also holds with $\mathfrak k'$ in place of $\mathfrak k$, but then $U$ lies in the 
narrower interval \eqref{6.4}.\par

By orthogonality and Chebychev's bound, 
\begin{equation}\label{6.6}
\int_{\mathscr L(U)}|g(\alpha)|\,\mathrm d\alpha \le \frac{U}{n} 
\int_0^1|g(\alpha)|^2\,\mathrm d\alpha \le UL. 
\end{equation}
This suggests the strategy of splitting the set $\mathscr L(U)$ into the subsets
\begin{align*}
\mathscr G&=\{\alpha\in \mathscr L(U): |f(\alpha)|^s\le P^sU^{-1}L^{-3}\},\\
\mathscr H&=\{\alpha\in{\mathscr L}(U): |f(\alpha)|^s > P^sU^{-1}L^{-3}\},
\end{align*}
since by \eqref{6.6}, we have
\begin{equation}\label{6.7}
\int_{\mathscr G}|g(\alpha)f(\alpha)^s|\,\mathrm d\alpha \ll P^sL^{-2}. 
\end{equation}
This leaves the set $\mathscr H$ for further discussion.\par

For $\alpha\in\mathscr H$ one has $|f(\alpha)|^t>P^t(UL^3)^{-t/s}$. Hence
\[
\int_{\mathscr H}|f(\alpha)|^s\,\mathrm d\alpha <P^{-t}U^{t/s}L^3 \int_0^1
|f(\alpha)|^{s+t}\,\mathrm d\alpha.
\]
Recalling our convention concerning the use of $\varepsilon$, we deduce via \eqref{6.2} 
that
\begin{equation}\label{6.8}
\int_{\mathscr H}|g(\alpha)f(\alpha)^s|\,\mathrm d\alpha
\ll U^{t/s -1}P^{s+\Delta_{s+t}+\varepsilon}.
\end{equation}
However, we have $U\ge U_0L^{-5}$, and furthermore $t/s-1$ is negative. Then by the 
second inequality in \eqref{5.5} we find that there is a positive number $\delta$ with 
\[
\int_{\mathscr H}|g(\alpha)f(\alpha)^s|\,\mathrm d\alpha
\ll U_0^{t/s -1}P^{s+\Delta_{s+t}+\varepsilon}\ll P^{s-\delta}.
\]
We now combine this bound with \eqref{6.5} and \eqref{6.7} and arrive at the final 
estimate
\begin{equation}\label{6.9}
\int_{\mathfrak k} |g(\alpha)f(\alpha)^s|\,\mathrm d\alpha \ll P^s L^{-1}, 
\end{equation}
valid when $\theta=5$, with the analogue for $\theta=4$ holding with $\mathfrak k'$ in 
place of $\mathfrak k$ subject to GRH.

\section{Pruning by height} We treat the intermediate arcs 
$\mathfrak K\setminus \mathfrak L$ and $\mathfrak K'\setminus \mathfrak L$ by a variant 
of the ideas developed in the previous section. We first slice these intermediate arcs into 
$O(L)$ pieces $\mathfrak P(Q)$, as defined in \eqref{4.5}, with $Q$ satisfying
\begin{equation}\label{7.1}
P^{1/5}\le Q\le \tfrac12 U_0^2. 
\end{equation}
Thus, when $\theta=5$, there is a value of $Q$ constrained by \eqref{7.1} with
\begin{equation}\label{7.2}
\int_{\mathfrak K\setminus\mathfrak L}|g(\alpha)f(\alpha)^s|\,\mathrm d\alpha 
\ll L\int_{\mathfrak P(Q)}|g(\alpha)f(\alpha)^s|\,\mathrm d\alpha .
\end{equation}
When $\theta=4$ the same conclusion holds, subject to GRH, with $\mathfrak K'$ replacing 
$\mathfrak K$.\par

We are ready to mimic the argument from the previous section, though we invoke Lemma 
\ref{lemma2.3} rather than Lemma \ref{lemma2.2}. We again suppose that the hypotheses 
in Lemma \ref{lemma5.1} are satisfied. In this setting the first inequality of \eqref{5.5} 
implies that for some positive number $\delta$, one has
\begin{equation}\label{7.3}
\int_{\mathfrak P(Q)} |f(\alpha)|^s\,\mathrm d\alpha \ll P^{s-k+\varepsilon} Q^{1-\delta}. 
\end{equation}
The role of the set $\mathscr T$ where $g(\alpha)$ is {tiny} is now played by the set
\[
{\mathscr S}=\{ \alpha\in\mathfrak P(Q): |g(\alpha)|\le n Q^{-1}\}.
\]
Thus, by virtue of \eqref{7.1} and \eqref{7.3}, we immediately have
\begin{equation}\label{7.4}
\int_{\mathscr S} |g(\alpha)f(\alpha)^s|\,\mathrm d\alpha\ll P^{s+\varepsilon}
Q^{-\delta} \ll P^{s-\delta/6}. 
\end{equation}

Next we observe that Lemma \ref{lemma4.2} supplies the bound 
$g(\alpha)\ll nL^4Q^{-1/2}$ for all $\alpha\in\mathfrak P(Q)$, provided that $Q$ satisfies 
\eqref{7.1}. Hence, the set of arcs $\mathfrak P(Q)$ is the union of $\mathscr S$ and 
$O(L)$ level sets
\begin{equation}\label{7.5}
\mathscr K(V)=\{\alpha\in\mathfrak P(Q): n/V\le |g(\alpha)|\le 2n/V\},
\end{equation}
with
\begin{equation}\label{7.6}
Q^{1/2}L^{-5}\le V\le Q, 
\end{equation}
at least for large $n$. In view of the bound \eqref{7.4}, there is consequently a value of 
$V$ satisfying \eqref{7.6} for which
\begin{equation}\label{7.7}
\int_{\mathfrak P(Q)} |g(\alpha)f(\alpha)^s|\,\mathrm d\alpha
\ll P^{s-\delta/6} + L \int_{{\mathscr K}(V)}  |g(\alpha)f(\alpha)^s|\,\mathrm d\alpha. 
\end{equation}

In the discussion of the previous section, pruning by size first removes from 
$\mathscr L(U)$ a portion $\mathscr G$ where $f(\alpha)$ was gentle enough to cooperate 
with the mean of $|g(\alpha)|$ over $\mathscr L(U)$, and the remaining set $\mathscr H$ 
required a slightly harder argument. In the present setting, again, the portion
\[
\mathscr E=\{\alpha\in{\mathscr K}(V): |f(\alpha)|^s\le P^sV^{-1}L^{-4}\}
\]
analogous to $\mathscr G$ will be equally easy, but its complement 
\[
\mathscr F=\{\alpha\in{\mathscr K}(V): |f(\alpha)|^s>P^sV^{-1}L^{-4}\}
\]
in $\mathscr K(V)$ is also straightforward to accommodate by arguing as in our earlier 
treatment of $\mathscr H$. Indeed, as in \eqref{6.7}, we now have
\begin{equation}\label{7.8}
\int_{\mathscr E} |g(\alpha)f(\alpha)^s|\,\mathrm d\alpha
\ll P^sL^{-3}. 
\end{equation}
Again following the initial steps of the treatment of $\mathscr H$, we find that
\[
\int_{\mathscr F} |f(\alpha)|^s\,\mathrm d\alpha \ll P^{-t}V^{t/s}L^4 
\int_{\mathfrak M(2Q)}|f(\alpha)|^{s+t}\,\mathrm d\alpha .
\]
Next, applying Lemma \ref{lemma2.3} together with \eqref{7.5}, we arrive at the bound
$$ \int_{\mathscr F}|g(\alpha)f(\alpha)^s|\,\mathrm d\alpha
\ll V^{t/s-1}P^{s+\varepsilon}Q^{2\Delta_{s+t}/k}. $$
This last bound is the counterpart of \eqref{6.8}. As before, we may suppose that $t/s-1$ is 
negative, whence \eqref{7.6} yields
\[
\int_{\mathscr F}|g(\alpha)f(\alpha)^s|\,\mathrm d\alpha
\ll P^{s+\varepsilon} Q^\omega,
\]
where
\[
\omega=\frac{1}{2}\Bigl(\frac{t}{s}-1\Bigr) +\frac{2\Delta_{s+t}}{k}
\le \frac{1}{2}\Bigl( \frac{t}{s}+\frac{\theta \Delta_{s+t}}{k}-1\Bigr).
\]
The second condition in \eqref{5.5} ensures that $\omega<0$. Hence, by \eqref{7.1}, it 
follows that there is a positive number $\delta>0$ for which
\begin{equation}\label{7.9}
\int_{\mathscr F}|g(\alpha)f(\alpha)^s|\,\mathrm d\alpha\ll P^{s-\delta}. 
\end{equation}

\par It remains to sum up the various contributions required to treat the integrals over 
$\mathfrak K\setminus \mathfrak L$ and $\mathfrak K'\setminus \mathfrak L$. By 
\eqref{7.7}, \eqref{7.8} and \eqref{7.9}, we see that
\[
\int_{\mathfrak P(Q)}|g(\alpha)f(\alpha)^s|\,\mathrm d \alpha \ll P^sL^{-2}.
\]
On substituting this estimate into \eqref{7.2}, we acquire the upper bound
\[
\int_{\mathfrak K\setminus\mathfrak L}|g(\alpha)f(\alpha)^s|\,\mathrm d\alpha \ll 
P^sL^{-1},
\]
valid when $\theta=5$, with the analogue for $\theta=4$ once again holding with 
${\mathfrak K'\setminus\mathfrak L}$ in place of ${\mathfrak K\setminus\mathfrak L}$ 
subject to GRH. Since $\mathfrak l=\mathfrak k\cup (\mathfrak K\setminus \mathfrak L)$, 
and likewise $\mathfrak l=\mathfrak k'\cup (\mathfrak K'\setminus \mathfrak L)$, this 
bound together with \eqref{6.9} proves Lemma \ref{lemma5.1}.  

\section{The theorems}
Our work performed thus far allows us to confirm the lower bound \eqref{1.2} subject to 
conditions that involve certain admissible exponents.

\begin{theorem}\label{thm8.1}
Let $k\ge 3$, and suppose that $s_0$, $t_0$ are real numbers satisfying both 
$0\le t_0\le s_0$ and the conditions \eqref{5.5} with $\theta=5$ and $(s,t) =(s_0,t_0)$. 
Then, whenever $s$ is a natural number with $s\ge \max\{s_0,k+3\}$, one has 
$r_{k,s}(n)\gg n^{s/k}/\log n$. If \eqref{5.5} holds only with $\theta=4$, then this lower 
bound for $r_{k,s}(n)$ holds subject to GRH.
\end{theorem}
   
\begin{proof} We choose $R=P^\eta$ with a value of $\eta$ satisfying $0<\eta<1/7$ 
chosen so small that Lemma \ref{lemma5.1} applies with $s=s_0$. Then, for $s\ge s_0$, 
we estimate $|f(\alpha)|^{s-s_0}$ trivially, and conclude via \eqref{3.2} that
\[
I(n,\mathfrak l)\ll n^{s/k}(\log n)^{-1}.
\]
Next, since $\mathfrak L=\mathfrak \grN\cup (\mathfrak L\setminus \mathfrak N)$, we find 
by combining Lemmata \ref{lemma3.1} and \ref{lemma4.3} via \eqref{3.2} that 
$I(n,\mathfrak L)\gg n^{s/k}$. Recalling that $I(n)=I(n,\mathfrak L) +I(n,\mathfrak l)$, the 
theorem is now immediate from \eqref{3.3}.
\end{proof}

We are now equipped to complete the proofs of the theorems presented in the introduction.

\begin{proof}[The proof of Theorems {\rm\ref{thm1.2}} and {\rm\ref{thm1.4}}] We verify 
the conditions in Theorem \ref{thm8.1} in the relevant cases by using explicit admissible 
exponents extracted from the tables in the appendices to \cite{VW, VW2000}, suitably 
rounded up, as we now explain.\par

Given an exponent $k$ with $5\le k\le 20$ and a value of $\theta\in \{4,5\}$, we fix values 
of $s_\theta$ and $t_\theta$ according to Table \ref{tab2}. We make use of the tables of 
permissible exponents $\lambda_u$ associated to each value of $k$ from the appendices of 
\cite{VW} and \cite{VW2000}, using the former for $k=5,6$, and the latter for 
$7\le k\le 20$. The admissible exponents $\Delta_u$ of the present paper are related to 
these tabulated values when $u$ is even via the relation
\[
\Delta_u=\lambda_{u/2}-u+k.
\]
When $u$ is odd, we may apply H\"older's inequality to interpolate between even values, so 
that
\[
\Delta_u=\tfrac{1}{2}\left( \lambda_{(u+1)/2}+\lambda_{(u-1)/2}\right) -u+k.
\]
When $u$ is either $s_\theta$ or $s_\theta+t_\theta$, admissible exponents calculated in 
this way are presented in Table \ref{tab2} to $4$ decimal places, rounded up in the final 
place. The final entries recorded in Table \ref{tab2} are computed values of the quantities
\[
\Omega_\theta =\frac{t_\theta}{s_\theta}+\frac{\theta\Delta_{s_\theta+t_\theta}}{k},
\]
again presented to $4$ decimal places rounded up in the final place. We thus see that both 
of the conditions \eqref{5.5} are met provided that $2\Del_{s_\theta}<k$ and 
$\Omega_\theta<1$, and such is the case for all entries of Table \ref{tab2}. By taking 
$s_0=s_\theta$ and $t_0=t_\theta$ in Theorem \ref{thm8.1}, we confirm the conclusion 
of Theorem \ref{thm1.2} in the case $\theta=5$, and likewise Theorem \ref{thm1.4} in the 
case $\theta=4$. This completes the proof of these theorems.
\end{proof}

\begin{table}[h]
\begin{center}
\begin{tabular}{ccccccccccc}
\toprule
$k$ & $s_4$  & $t_4$  & $\Delta_{s_4}$ & $\Delta_{s_4+t_4}$ & $\Omega_4$ & $s_5$ & $t_5$ & $\Delta_{s_5}$ & $\Delta_{s_5+t_5}$ & $\Omega_5$\\
\toprule
$5$ &  $8$  & $4$ & $1.4387$ & $0.5418$ & $0.9335$ & & & & & \\
$6$ & $10$ & $4$ & $1.7247$ & $0.8506$ & $0.9671$ & $11$ & $3$ & $1.4782$ &  $0.8516$ & $0.9816$ \\
$7$ & $12$ & $6$ & $2.0144$ & $0.8470$ & $0.9840$ & $13$ & $5$ & $1.7778$ & $0.8470$ & $0.9897$ \\
$8$ & $14$ & $6$ & $2.3106$ & $1.1284$ & $0.9928$ & $16$ & $8$ & $1.8429$ & $0.6562$ & $0.9102$\\
$9$ & $17$ & $7$ & $2.3733$ & $1.1293$ & $0.9137$ & $18$ & $8$ & $2.1426$ & $0.8960$ & $0.9422$ \\
$10$ & $19$ & $7$ & $2.6661$ & $1.3944$ & $0.9262$ & $20$ & $10$ & $2.4376$ & $0.9304$ & $0.9652$ \\
$11$ & $21$ & $9$ & $2.9571$ & $1.3941$ & $0.9356$ & $22$ & $10$ & $2.7293$ & $1.1641$ & $0.9837$ \\
$12$ & $23$ & $9$ & $3.2438$ & $1.6487$ & $0.9409$ & $24$ & $12$ & $3.0175$ & $1.1896$ & $0.9957$ \\
$13$ & $25$ & $9$ & $3.5299$ & $1.9063$ & $0.9466$ & $27$ & $13$ & $3.0997$ & $1.2191$ & $0.9504$ \\
$14$ & $27$ & $11$ & $3.8147$ & $1.8961$ & $0.9492$ & $29$ & $13$ & $3.3840$ & $1.4408$ & $0.9629$ \\
$15$ & $29$ & $11$ & $4.0984$ & $2.1484$ & $0.9523$ & $31$ & $15$ & $3.6673$ & $1.4665$ & $0.9728$ \\
$16$ & $31$ & $13$ & $4.3819$ & $2.1408$ & $0.9546$ & $33$ & $15$ & $3.9503$ & $1.6884$ & $0.9822$ \\
$17$ & $32$ & $14$ & $4.8887$ & $2.3897$ & $0.9998$ & $35$ & $17$ & $4.2327$ & $1.7118$ & $0.9892$ \\
$18$ & $34$ & $14$ & $5.1707$ & $2.6416$ & $0.9988$ & $37$ & $19$ & $4.5147$ & $1.7385$ & $0.9965$ \\
$19$ & $36$ & $16$ & $5.4518$ & $2.6301$ & $0.9982$ & $40$ & $18$ & $4.5867$ & $1.9556$ & $0.9647$ \\
$20$ & $38$ & $16$ & $5.7323$ & $2.8788$ & $0.9969$ &  $42$ & $20$ & $4.8667$ & $1.9802$ & $0.9713$ \\
\bottomrule
\end{tabular}\\[6pt]
\end{center}
\caption{Choice of exponents for $5\le k\le 20$.}\label{tab2}
\end{table}

The proof of Theorems \ref{thm1.1} and \ref{thm1.3} is more involved, and this will 
require us to establish a technical result along the way. Again, we aim to verify the 
conditions of Theorem \ref{thm8.1}. The condition $2\Delta_s<k$ in \eqref{5.5} is only 
a moderate constraint on $s$. Indeed, writing $\sigma_0=\frac12+\log 2$, one finds from 
\eqref{2.1} that $\Eta(\sigma_0)=\frac12$. Hence, there is an even integer $s_0$ in the 
interval $(\sigma_0k,\sigma_0k +2]$, and we have $\Eta(s_0/k)<\frac12$. By Lemma 
\ref{lemma2.2} with $t=s_0$, it follows that whenever $s\ge s_0$, there is an admissible 
exponent $\Delta_s<k/2$, as required. In particular, we now have the first inequality in 
\eqref{5.5} whenever $s\ge \sigma_0+2$.\par

The second inequality in \eqref{5.5} is the more restrictive condition. Throughout, let 
$\theta=4$ or $5$. Note that whenever $s+t$ is an even integer, then by Lemma 
\ref{lemma2.2} we may assume that the exponent $\Delta_{s+t}=k\Eta((s+t)/k)$ is 
admissible. With this choice of $\Delta_{s+t}$ we write
\[
s=\sigma k,\qquad t=\tau k,
\]
and then have 
\begin{equation}\label{8.1}
\frac{t}{s}+\frac{\theta \Delta_{s+t}}{k}=\frac{\tau}{\sigma}+\theta\Eta(\sigma+\tau). 
\end{equation}
For a given value of $s$ one wishes to minimize this expression, and hence one has to 
choose $\tau$ optimally, or at least nearly so. The constraint $0\le t\le s$ translates to 
$0\le \tau\le \sigma$. We temporarily ignore that \eqref{8.1} is available only for even 
values of $s+t$, a complication soon to be resolved in a technical lemma. Instead, we 
compute the function $E_\theta:[\frac54, 3]\to\mathbb R$ defined by
\begin{equation}\label{8.2}
E_\theta (\sigma)=\min_{0\le\tau\le \sigma}h(\tau),
\end{equation}
where for each fixed $\sigma\in[\frac54,3]$, the function $h:[0,\infty)\to\mathbb R$ is 
defined by 
\begin{equation}\label{8.3}
h(\tau)=\frac{\tau}{\sigma} + \theta\Eta(\sigma+\tau). 
\end{equation}
Notice that in naming the function $h$, we have suppressed its dependence on both 
$\theta$ and $\sigma$ in order to simplify our exposition.\par

The function $h$ is continuously differentiable, and we find from \eqref{2.2} that
\[
h'(\tau)=\frac1{\sigma}-\frac{\theta\Eta(\sigma+\tau)}{1+\Eta(\sigma+\tau)}.
\]
Note that $h'(\tau) <0$ if and only if $\Eta(\sigma+\tau)> 1/(\theta\sigma -1)$. In 
particular, we see from Lemma \ref{lemma2.1} that $h'(0)$ is negative, so that the function 
$h$ is decreasing for small values of $\tau$. Also, the equation $h'(\tau)=0$ is equivalent to
\begin{equation}\label{8.4}
\Eta(\sigma+\tau)=\frac1{\theta\sigma-1}. 
\end{equation}
Since $ \Eta(\sigma+\tau)$ is continuous in $\tau$ and strictly decreases from  
$\Eta(\sigma)$ to $0$ as $\tau$ runs through positive real values, we conclude from 
Lemma \ref{lemma2.1} that \eqref{8.4} has exactly one solution that we denote by 
$\tau(\sigma)$. We apply \eqref{2.1} to $\Eta(\sigma+\tau(\sigma))$ and insert the 
identity \eqref{8.4}, thus finding that
\begin{equation}\label{8.5}
\tau(\sigma)=1-\sigma-\frac1{\theta\sigma-1}+\log(\theta\sigma-1). 
\end{equation}
In view of \eqref{8.3}, the function $h(\tau)$ is increasing for large $\tau$, whence 
$h(\tau(\sigma))$ is the minimum value of $h$. Further, since $h'(0)<0$, we must have 
$\tau(\sigma)>0$.\par

With \eqref{8.2} in mind, we now show that one has $\tau(\sigma)<\sigma$. By 
\eqref{8.5}, this inequality holds if and only if
\begin{equation}\label{8.6}
1+ \log(\theta\sigma-1)<2\sigma +\frac1{\theta\sigma-1}.
\end{equation}
This is readily checked numerically for $\sigma=\frac54$, and amounts to checking that
\[
2.658\ldots =1+\log \Bigl( \frac{21}{4}\Bigr)<\frac{5}{2}+\frac{1}{21/4}=2.690\ldots .
\]
Meanwhile, for $\sigma>\frac54$, the derivative of the left hand side of \eqref{8.6} is 
smaller than the derivative of the right hand side. Again, this is easily checked. Thus we 
discern that the upper bound $\tau(\sigma)<\sigma$ holds for $\sigma\in[\frac54,3]$. 
Consequently, by \eqref{8.2} and \eqref{8.4},
\begin{equation}\label{8.7}
E_\theta(\sigma)=\frac{\tau(\sigma)}{\sigma} + \theta\Eta(\sigma+\tau(\sigma))
=\frac{\tau(\sigma)}{\sigma} + \frac{\theta}{\theta\sigma-1}. 
\end{equation}

Next we note that $\tau(\sigma)$ is decreasing on the interval $[\frac32, 3]$. Indeed, it 
follows from \eqref{8.5} that throughout the latter interval, the derivative
\begin{equation}\label{8.8}
\tau'(\sigma)=-1+\frac{\theta}{\theta\sigma-1}+\frac{\theta}{(\theta \sigma -1)^2}
\end{equation}
is negative. Since $1/\sigma$ also decreases, we see from \eqref{8.7} that 
$E_\theta(\sigma)$ is strictly decreasing on $[\frac32,3]$. One may check numerically from 
\eqref{8.5} and \eqref{8.7} that
\[
E_\theta(3)<1<E_\theta(3/2).
\]
Hence there is a unique number $c_\theta$ with $E_\theta(c_\theta)=1$. If we now 
eliminate $\tau(c_\theta)$ between \eqref{8.5} and \eqref{8.7}, then we obtain
\[
1-c_\theta-\frac{1}{\theta c_\theta -1}+\log (\theta c_\theta -1)
=c_\theta -\frac{\theta c_\theta }{\theta c_\theta -1}.
\]
This equation shows that $c_5$ is the real number $c$ occuring in the statement of 
Theorem \ref{thm1.1}, and that $c_4$ is the real number $c'$ occuring in the statement of 
Theorem \ref{thm1.3}.\par

We now attend to the key technical lemma previously advertised.

\begin{lemma}\label{lemma8.2}
Let $\theta\in \{4,5\}$ and $k\ge 17$. Then there exists a number 
$\sigma_\theta=\sigma_{\theta,k}$, with $\sigma_\theta \in (c_\theta,c_\theta+4/k)$, 
satisfying the condition that $k(\sigma_\theta+\tau(\sigma_\theta))$ is an even integer.
\end{lemma}

\begin{proof} We begin by noting from \eqref{8.8} that the function 
$k(\sigma+\tau(\sigma))$ is increasing on $(c_\theta,c_\theta+4/k)$. It follows that this 
function maps the latter interval onto
\[
\left( kc_\theta+k\tau(c_\theta), kc_\theta +4+k\tau(c_\theta+4/k)\right) .
\]
Provided that the inequality
\begin{equation}\label{8.9}
k(\tau(c_\theta)-\tau(c_\theta+4/k))<2 
\end{equation}
holds, we see that the set of values of $k\sigma+k\tau(\sigma)$, for 
$\sigma\in (c_\theta,c_\theta+4/k)$, covers an interval of length exceeding $2$, and 
therefore contains an even integer $2l$, say. One then has 
$2l=k(\sigma_\theta+\tau(\sigma_\theta))$ for some 
$\sigma_\theta\in (c_\theta,c_\theta+4/k)$, as desired.\par

It remains to check \eqref{8.9}. By the mean value theorem and \eqref{8.8}, there exists 
a real number $\sigma \in (c_\theta,c_\theta+4/k)$ with 
\[
k\left( \tau(c_\theta )-\tau(c_\theta +4/k)\right) =-4\tau'(\sigma )=
4-\frac{4\theta }{\theta \sigma -1}-\frac{4\theta }{(\theta \sigma -1)^2}.
\]
However, we have
\[
\frac{\theta }{\theta \sigma-1}+\frac{\theta }{(\theta \sigma-1)^2} >\frac12
\]
provided only that
\[
(\theta \sigma-1-\theta)^2<\theta^2+2\theta ,
\]
and this is assured whenever
\[
1+1/\theta-\sqrt{1+2/\theta}<\sigma<1+1/\theta+\sqrt{1+2/\theta}.
\]
When $\theta=5$, this last constraint is satisfied for $\sigma\in (0.017,2.383)$, and hence 
for $c<\sigma <c+4/k$ when $k\ge 17$. When $\theta=4$, meanwhile, this last constraint 
is satisfied for $\sigma\in (0.026,2.474)$, and hence with ease for $c'<\sigma <c'+4/k$ 
under the same condition on $k$. This establishes \eqref{8.9} for $k\ge 17$, and completes 
the proof of the lemma.
\end{proof}

Finally, we establish Theorems \ref{thm1.1} and \ref{thm1.3}, dividing the natural numbers 
$k$ into three ranges. In the first range $k\ge 17$, we check the conditions of Theorem 
\ref{thm8.1} when $\theta\in \{4,5\}$. Let $s_\theta =\sigma_\theta k$ and 
$t_\theta=\tau(\sigma_\theta)k$. Then $s_\theta +t_\theta$ is the even integer provided 
by Lemma \ref{lemma8.2}. Further, since $E_\theta(\sigma)$ is a strictly decreasing 
function on $[\tfrac{3}{2},3]$ and 
\[
3>c_\theta+4/k>\sigma_\theta>c_\theta>3/2,
\]
we have $E_\theta (\sigma_\theta )<E_\theta(c_\theta)=1$. By \eqref{8.1}, \eqref{8.2} 
and \eqref{8.3} we conclude that the second inequality in \eqref{5.5} holds with 
$s=s_\theta$ and $t=t_\theta$. By taking $(s_0,t_0)=(s_\theta, t_\theta)$ in Theorem 
\ref{thm8.1}, therefore, we obtain Theorem \ref{thm1.1} when $\theta=5$, and Theorem 
\ref{thm1.3} when $\theta=4$.\par

In the second range $\theta+1\le k\le 16$, the conclusions of Theorems \ref{thm1.1} and 
\ref{thm1.3} follow, respectively, from Theorems \ref{thm1.2} and \ref{thm1.4}. The third 
range $1\le k<\theta$, meanwhile, is more or less trivial. Indeed, our remarks concerning 
the naive decoupling approach in the preamble to the statement of Theorem \ref{thm1.2} 
already establish Theorem \ref{thm1.1} in the case $k=4$, and the work of Kawada 
\cite[Theorem 2]{Kaw1996} confirms both Theorem \ref{thm1.1} and \ref{thm1.3} when 
$k=3$. This leaves the case $k=2$ handled by Hooley \cite{HooActa, HooTract} and Linnik 
\cite{L1, L2}, and the trivial case $k=1$. With these observations, all of the loose ends 
associated with the confirmation of Theorems \ref{thm1.1} and \ref{thm1.3} have been 
tied up. 

\bibliographystyle{amsbracket}
\providecommand{\bysame}{\leavevmode\hbox to3em{\hrulefill}\thinspace}

\end{document}